\documentclass{amsart}
\usepackage{amscd}
\usepackage{amsfonts}
\usepackage{amsmath}
\usepackage{amssymb}
\usepackage{amsthm}
\usepackage{fancyhdr}
\usepackage{latexsym}

\newtheorem{theorem}{Theorem}
\newtheorem{lemma}[theorem]{Lemma}

\newcommand{\R}{\mathbb R}

\newcommand{\p}{\partial}

\begin{document}

\title[Ill-posedness for the FORQ equation]{Ill-posedness in the critical Sobolev space for the Fokas-Olver-Rosenau-Qiao equation}

\author{Dan-Andrei Geba}

\address{Department of Mathematics, University of Rochester, Rochester, NY 14627, U.S.A.}
\email{dangeba@math.rochester.edu}
\date{}

\begin{abstract}
This article proves norm inflation in the critical Sobolev space  $H^{5/2}(\R)$ for  the Fokas-Olver-Rosenau-Qiao equation, which is a modified Camassa-Holm-type equation with cubic nonlinearity. This result complements the well-posedness theory for this equation, which was previously known to be locally well-posed in $H^{s}(\R)$ for $s>5/2$. The proof relies on the construction of explicit initial data satisfying the previously known blow-up criteria for the Fokas-Olver-Rosenau-Qiao  equation, a step that appears to be of independent interest.
\end{abstract}


\maketitle


\section{Introduction}

In this paper, we study the Cauchy problem for the Fokas-Olver-Rosenau-Qiao (FORQ) equation on the real line:
\begin{equation}
\begin{cases}
u_t-u_{txx}+3u^2u_x-u_x^3-4uu_xu_{xx}+2u_xu_{xx}^2-u^2u_{xxx}+u_x^2u_{xxx}=0,\\
u(x,0)=u_0(x), \qquad x\in \R,
\end{cases}
\label{cp}
\end{equation}
and show that it is ill-posed (IP) in the Sobolev space $H^{5/2}(\R)$.

Introducing the momentum variable
\[
m=u-u_{xx},
\]
the FORQ equation can be written in the local form
\begin{equation}
m_t+\bigl((u^2-u_x^2)m\bigr)_x=0.
\label{floc}
\end{equation}
It also admits the nonlocal formulation
\begin{equation}
u_t+u^2u_x-\frac{1}{3}u_x^3
+\frac{1}{3}(1-\p_x^2)^{-1}{u_x^3}
+\frac{1}{3}\p_x(1-\p_x^2)^{-1}
\left\{\frac{2}{3}u^3+uu_x^2\right\}=0.\label{fnonloc}
\end{equation}

The FORQ equation was independently derived by Fokas \cite{F95}, Olver and Rosenau \cite{OR96}, Fuchssteiner \cite{F96}, and Qiao \cite{Q06}. Fokas introduced the equation as an integrable generalization of the modified Korteweg-de Vries equation, whereas Qiao obtained it through an approximation procedure for the two-dimensional Euler equations. Throughout this work, by \emph{integrable} we mean that the equation possesses infinitely many conservation laws, admits a Lax pair and a bi-Hamiltonian formulation, and can be solved by the inverse scattering method.

The FORQ equation also belongs to a family of integrable Camassa-Holm-type equations with cubic nonlinearities introduced by Novikov \cite{N09}. A notable member of this family is the Novikov equation,
\begin{equation}
u_t+u^2u_x
+\frac{1}{2}(1-\p_x^2)^{-1}{u_x^3}
+\frac{3}{2}\p_x(1-\p_x^2)^{-1}
\left\{\frac{2}{3}u^3+uu_x^2\right\}=0.\label{ne-nonloc}
\end{equation}
For comparison, the classical Camassa-Holm equation, which contains quadratic nonlinearities, can be written as
\begin{equation*}
m_t+um_x+2u_xm=0,
\qquad m=u-u_{xx}.
\end{equation*}
Comparing this expression with \eqref{floc} highlights the structural similarities between the Camassa-Holm and the FORQ equations.

A remarkable feature shared by the FORQ, Novikov, and Camassa-Holm equations is the existence of peaked traveling-wave solutions, known as \emph{peakons}, on both the real line and the circle. For the FORQ equation on $\R$, the peakon solution is given by
\[
u_c(t,x)=\sqrt{\frac{3c}{2}}\,e^{-|x-ct|},
\]
where $c>0$ denotes the wave speed.

The well-posedness (WP) theory in the sense of Hadamard - namely, existence, uniqueness, and continuous dependence on the initial data - for nonlinear evolution equations remains one of the most active and challenging areas of research in the theory of partial differential equations. Camassa-Holm-type equations belong to this class, and the WP theory of the FORQ equation has attracted considerable attention.

Fu, Gui, Liu, and Qu \cite{Liu13-2} established local WP in $H^s(\R)$ for $s>5/2$ and in the critical Besov space $B^{5/2}_{2,1}(\R)$. They also analyzed blow-up scenarios and proved the nonexistence of smooth traveling-wave solutions. In a subsequent work, Gui, Liu, Olver, and Qu \cite{Liu13} investigated singularity formation and showed that singularities can occur only through wave breaking. Later, Himonas and Mantzavinos \cite{HM14} provided an alternative proof of local WP in $H^s(\R)$ for $s>5/2$ and extended the result to the periodic setting. They further established that the data-to-solution map fails to be uniformly continuous in both $H^s(\R)$ and $H^s(\mathbb{T})$ whenever $s<3/2$ or $s>5/2$. Given the quasilinear nature of the FORQ equation, this phenomenon is not unexpected. Subsequently, Himonas and Holliman \cite{HH19} showed that uniqueness fails on both the line and the circle for $s<3/2$. Taken together, these results leave the range
\[
\frac{3}{2}\le s\le \frac{5}{2}
\]
as a gap in the known WP theory for the FORQ equation.

Our main result closes the endpoint $s=5/2$ by proving norm inflation in $H^{5/2}(\R)$. Consequently, the data-to-solution map cannot be continuous at this regularity, establishing $s=5/2$ as the critical Sobolev index for the WP of the Cauchy problem \eqref{cp}. More precisely, we prove the following.

\begin{theorem}
For every $\epsilon>0$, there exists initial data $u_0\in H^\infty(\R)$ satisfying
\[
\|u_0\|_{H^{5/2}(\R)}\leq \epsilon
\]
such that the Cauchy problem \eqref{cp} admits a unique solution
\[
u\in C([0,T);H^\infty(\R))
\]
with maximal lifespan $T<\epsilon$ and
\[
\limsup_{t\to T^-}\|u(t)\|_{H^{5/2}(\R)}=\infty.
\]
\label{main}
\end{theorem}

A natural direction for future work is to determine whether an analogous norm-inflation phenomenon occurs in the periodic setting. Another important question is to identify the mechanism responsible for the IP of \eqref{cp} in the range $3/2\le s<5/2$, whether it arises from nonuniqueness, failure of continuity of the data-to-solution map, or some other phenomenon.

It is also instructive to compare the WP theory of the FORQ equation with those of the Camassa-Holm and Novikov equations. Remarkably, both equations have the critical Sobolev index $s=3/2$ on the line and on the circle, as recently established by Guo, Liu, Molinet, and Yin \cite{GLM19}. For the Camassa-Holm equation, this is not particularly surprising since its quadratic nonlinearity is generally more tractable than the cubic nonlinearity of the FORQ equation. By contrast, the corresponding result for the Novikov equation is somewhat unexpected. Indeed, the nonlocal formulations \eqref{fnonloc} and \eqref{ne-nonloc} of the FORQ and Novikov equations differ, for all practical purposes, only by the additional term $-\frac{1}{3}u_x^3$ appearing in the FORQ equation.
 
 Next, we comment on the strategy underlying our main result. A careful reading of Guo et al. \cite{GLM19} reveals a general roadmap for establishing IP through norm inflation for certain nonlinear evolution equations. The essential ingredients are: (i) a comprehensive WP theory, (ii) a collection of suitable energy estimates (typically obtained as part of the WP analysis), and (iii) a blowup result with sufficiently precise control of the blowup time. While the first two components are by now relatively standard, the blowup analysis is often considerably more delicate. Such results are usually formulated in terms of a set of conditions on the initial data that guarantee finite-time blowup. What is sometimes lacking, however, is an explicit example of initial data satisfying these conditions. This was the case for the Novikov equation until Guo et al. \cite{GLM19} constructed concrete blowup data. The situation for the FORQ equation is similar: both Fu et al. \cite{Liu13-2} and Gui et al. \cite{Liu13} provide only sufficient conditions for blowup. Beyond the main theorem, a further contribution of the present work is the construction of explicit blowup data for the FORQ equation.

The paper is organized as follows. In Section 2, we review known results concerning WP and finite-time blowup for the FORQ equation, and collect several technical lemmas that will be used throughout the proof. In Section 3, we construct the initial data that lead to blowup and, ultimately, to norm inflation. Finally, in Section 4, we combine the results of the previous sections to establish the main theorem by contradiction.

Throughout the paper, we use the standard notation $A\lesssim B$ to denote the inequality $A\le cB$, where $c>0$ is a constant. Likewise, $A\sim B$ means that both $A\lesssim B$ and $B\lesssim A$ hold. Finally, $A\ll B$ indicates that $A\le cB$ for a sufficiently small constant $c>0$.


\section{Toolbox}

To state the relevant known results for the FORQ equation, we first recall the definitions of Besov and Sobolev spaces. Let $\mathcal{S}'(\R)$ denote the space of tempered distributions on $\R$, and define the Fourier transform of a function $f:\R\to\mathbb{C}$ by
\begin{equation*}
\widehat{f}(\xi):= \int_\R e^{-ix\xi} f(x)\,dx.
\end{equation*}
Let $\varphi: \R\to [0,1]$ to be a smooth, radially decreasing function such that
\begin{equation*}
\varphi(\xi)=\begin{cases}
1, \quad |\xi|\leq 5/4,\\
0, \quad |\xi|\geq 8/5.
\end{cases}
\end{equation*}
For each integer $k\geq 0$, we define the Littlewood-Paley projector $P_k$ by
\begin{equation*}
\widehat{P_0f}(\xi):=\varphi(\xi) \widehat{f}(\xi), \qquad \widehat{P_kf}(\xi):=\left(\varphi(\xi/2^k)- \varphi(\xi/2^{k-1})\right) \widehat{f}(\xi), \quad \forall\,k\geq 1.
\end{equation*}
The Besov space $B^s_{p,q}(\R)$ is then defined by
\begin{equation*}
B^s_{p,q}(\R):=\left\{f\in \mathcal{S}'(\R);\ \|f\|_{B^s_{p,q}}:= \left\|2^{ks}\|P_k f\|_{L^p}\right\|_{l^q_k}<\infty\right\}.
\end{equation*}
It is well known that the Sobolev space 
\begin{equation*}
H^s(\R):=\left\{f\in \mathcal{S}'(\R);\ \|f\|_{H^s}:= \left(\int_\R (1+|\xi|^2)^s|\widehat{f}(\xi)|^2\,d\xi\right)^{1/2}<\infty\right\}.
\end{equation*}
coincides with the Besov space $B^s_{2,2}(\R)$.

We now recall the main WP results for the FORQ equation obtained in \cite{Liu13-2}, \cite{Liu13}, and \cite{HM14}:

\begin{theorem}
Let $s>5/2$ and $u_0\in H^s(\R)$. Then there exists a maximal time $T\geq \overline{T}(\|u_0\|_{H^{5/2+}})>0$, and a unique solution 
\begin{equation*}
u\in C([0,T); H^s(\R))\cap C^1([0,T); H^{s-1}(\R))
\end{equation*}
to the Cauchy problem \eqref{cp}, which depends continuously on the initial data $u_0$.

The lifespan $T$ is independent of the Sobolev index $s$ of the initial data $u_0$. In particular, if $u_0\in H^\infty(\R)$, then $u\in C([0,T); H^\infty(\R))$. Finally, if $T<\infty$, then
\begin{equation*}
\limsup_{t\to T^-}\|u(t)\|_{H^s}=\infty.
\end{equation*}
\label{wp-th} 
\end{theorem}

Next, we recall a finite-time blowup criterion established in \cite{Liu13}.

\begin{theorem}
Let $w_0\in H^s(\R)$ with $s>5/2$. Suppose there exists $x_0\in \R$ such that
\[
z_{0}:=(1-\partial^2_x)w_{0}\geq 0\ \text{on}\ \R, \qquad z_{0}(x_0)>0,
\]
and
\[
w'_0(x_0)<-\left(\frac{\sqrt{2}\|w_0\|^3_{H^1}}{z_{0}(x_0)}\right)^{1/2}.
\]
Then the solution to the Cauchy problem \eqref{cp} with $u_0=w_0$ blows up in finite time $T$, with the upper bound 
\begin{equation*}
T\leq \,\frac{-w'_0(x_0)}{\sqrt{2}\|w_0\|^3_{H^1}}- \frac{1}{2}\sqrt{\left(\frac{\sqrt{2}w'_0(x_0)}{\|w_0\|^3_{H^1}}\right)^2-\frac{2\sqrt{2}}{\|w_0\|^3_{H^1}\,z_0(x_0)}}.
\end{equation*}
\label{blow-th}
\end{theorem}

Finally, we record two technical lemmas from \cite{GLM19}. The first combines Besov and Sobolev norms.
\begin{lemma}
If $w\in H^2(\R)$, then 
\begin{equation*}
\|w_x\|_{L^\infty}\lesssim 1+\|w\|_{B^1_{\infty,\infty}}\log_2(2+\|w\|^2_{H^2}).
\end{equation*}
holds true.
\label{wx-li}
\end{lemma}
The second is a logarithmic Gronwall-type inequality. 
\begin{lemma}
Let $A\in C^1([0,T);\R^+)$, $0<T\leq \infty$, satisfy 
\[
\frac{d}{dt} A(t)\leq B\, A(t)\ln(2+A(t)), \qquad \forall\,t\in [0,T),
\]
for some constant $B>0$. Then
\[
A(t)\leq (2+A(0))^{e^{Bt}}, \qquad \forall\,t\in [0,T).
\] 
\label{gron}
\end{lemma}


\section{Construction of the initial data}
This construction is inspired by the one introduced for the Novikov equation in \cite{GLM19}. Let $\phi$ be a fixed nonnegative even Schwartz function satisfying $\phi(0)=1$ (e.g., $\phi(x)=e^{-x^2}$), and define
\[
a_j(x):=\frac1j\,\phi(2^j x), \quad \forall\ j\geq 1, \ x\in \R.
\] 
Let also $\psi\in C_c^\infty(\R)$ be a fixed nonnegative function with supp\, $\psi\subseteq (1,2)$. Using these functions, we define a sequence of initial data $(u_{0,K})_{K\geq 1}$ by
\begin{equation*}
u_{0,K}(x):= (1-\partial_x^2)^{-1}\left(\psi(x)+\sum_{1\leq j\leq K} a_j(x)\right),  \quad \forall\ K\geq 1, \ x\in \R.
\end{equation*}
Setting 
\begin{equation}
y_{0,K}(x):= (1-\partial_x^2)u_{0,K}(x)=\psi(x)+\sum_{1\leq j\leq K} a_j(x),  \quad \forall\ K\geq 1, \ x\in \R,\label{y0K}
\end{equation}
we see that each $y_{0,K}$ is a real-valued, nonnegative Schwartz (and therefore belongs to $H^\infty$) function. Since $\psi(0)=0$ and $\phi(0)=1$,
\begin{equation*}
y_{0,K}(0)= \sum_{1\leq j\leq K} \frac1j, \quad \forall\ K\geq 1,
\end{equation*}
and consequently
\begin{equation}
\lim_{K\to \infty} y_{0,K}(0)= \infty.\label{yi}
\end{equation}

Since 
\[
u_{0,K}=G*y_{0,K},
\]
where
\[
G(x)=\frac12 e^{-|x|}
\]
is the Green's function of $(1-\partial_x^2)^{-1}$ on $\mathbb R$, each $u_{0,K}$ is also a real-valued, nonnegative $H^\infty$ function. Differentiating the convolution formula gives
\[\aligned
u'_{0,K}(0)&= \frac{1}{2}\int_0^\infty e^{-\xi}\,(y_{0,K}(-\xi)- y_{0,K}(\xi))\,d\xi\\
&= \frac{1}{2}\int_0^\infty e^{-\xi}\,(\psi(-\xi)- \psi(\xi))\,d\xi=- \frac{1}{2}\int_1^2 e^{-\xi}\, \psi(\xi)\,d\xi.
\endaligned
\]
Here we used the evenness of the functions $a_j$ and the support condition on $\psi$. Since $\psi$ is nonnegative and is not identically zero, 
\begin{equation*}
-u'_{0,K}(0)\gtrsim 1,
\end{equation*} 
uniformly in $K$. Combining this  with \eqref{yi}, we obtain
\begin{equation}
\lim_{K\to\infty} u_{0, K}'(0)y_{0, K}(0)=-\infty,\label{lim-1}
\end{equation}
and
\begin{equation}
\lim_{K\to\infty} (u_{0, K}'(0))^2y_{0, K}(0)=\infty,\label{lim-2}
\end{equation}
The next lemma shows that the sequence $(u_{0,K})_{K\geq 1}$ remains uniformly bounded in $H^{5/2}$.

\begin{lemma}
The estimate 
\begin{equation}
\|u_{0, K}\|_{H^{5/2}}\lesssim 1\label{52}
\end{equation}
holds uniformly in $K$.
\end{lemma}
\begin{proof}
Since 
\[
\|u_{0, K}\|_{H^{5/2}}\sim \|y_{0, K}\|_{H^{1/2}},
\]
it suffices, by \eqref{y0K}, to prove that
\begin{equation*}
\left\|\sum_{1\leq j\leq K} a_j\right\|_{H^{1/2}}\lesssim 1\label{12}
\end{equation*}
uniformly in $K$. 

To this end, we claim that
\begin{equation}
\left|\langle a_j, a_k\rangle_{H^{1/2}}\right|\lesssim\frac{2^{-|j-k|}}{jk}, \quad\forall\ 1\leq j,\,k.\label{jk}
\end{equation}
Indeed, assuming \eqref{jk},
\begin{equation*}
\left\|\sum_{1\leq j\leq K} a_j\right\|^2_{H^{1/2}}= \sum_{1\leq j,\,k\leq K} \langle a_j, a_k\rangle_{H^{1/2}}\lesssim \sum_{1\leq j,\,k\leq K} \frac{2^{-|j-k|}}{jk}\leq \left\|\frac1p\right\|^2_{l^2_p}\, \|2^{-|q|}\|_{l^1_q}\lesssim 1,
\end{equation*}
where we used Cauchy-Schwarz followed by Young's inequality.

It therefore remains to establish \eqref{jk}. By symmetry,  we may assume $j\geq k$. A direct computation yields
\begin{equation*}
\left|\langle a_j, a_k\rangle_{H^{1/2}}\right|\sim \frac{1}{jk\,2^k}\left|\int_\R (1+2^{2j}|\eta|^2)^{1/2}\,\widehat{\phi}(\eta)\,\overline{\widehat{\phi}(2^{j-k}\eta)}\,d\eta\right|.
\end{equation*}
Since $\phi$ is a Schwartz function, we estimate the integral over the three regions
\[
1\leq |\eta|, \qquad 2^{k-j}\leq |\eta|\leq 1, \qquad |\eta|\leq 2^{k-j},
\]
obtaining
\[\aligned
\frac{1}{jk\,2^k}&\left|\int_{1\leq |\eta|} (1+2^{2j}|\eta|^2)^{1/2}\,\widehat{\phi}(\eta)\,\overline{\widehat{\phi}(2^{j-k}\eta)}\,d\eta\right|\\
&\qquad\qquad\lesssim \frac{1}{jk\,2^k}\int_{|\eta|>1} 2^{j}|\eta|\,\frac{1}{|\eta|}\,\frac{1}{2^{2(j-k)}|\eta|^2}\,d\eta\sim \frac{2^{-(j-k)}}{jk},
\endaligned
\]
\[\aligned
\frac{1}{jk\,2^k}&\left|\int_{2^{k-j}\leq |\eta|\leq 1} (1+2^{2j}|\eta|^2)^{1/2}\,\widehat{\phi}(\eta)\,\overline{\widehat{\phi}(2^{j-k}\eta)}\,d\eta\right|\\
&\qquad\qquad\lesssim \frac{1}{jk\,2^k}\int_{2^{k-j}\leq |\eta|\leq 1} 2^{j}|\eta|\,\frac{1}{2^{3(j-k)}|\eta|^3}\,d\eta\sim \frac{2^{-(j-k)}}{jk},\endaligned
\]
and
\[\aligned
\frac{1}{jk\,2^k}&\left|\int_{|\eta|\leq 2^{k-j}} (1+2^{2j}|\eta|^2)^{1/2}\,\widehat{\phi}(\eta)\,\overline{\widehat{\phi}(2^{j-k}\eta)}\,d\eta\right|\\
&\qquad\qquad\lesssim \frac{1}{jk\,2^k}\int_{|\eta|\leq 2^{k-j}} (1+2^{j}|\eta|)\,d\eta\lesssim \frac{2^{-(j-k)}}{jk}.\endaligned
\]
These estimates prove\eqref{jk}, completing the proof. 
\end{proof}

Combining \eqref{lim-2} with \eqref{52}, we conclude that
\begin{equation}
\left(u_{0, K}'(0)\right)^2y_{0, K}(0)>\sqrt{2}\|u_{0,K}\|^3_{H^1}\label{ineq-1}
\end{equation}
for all sufficiently large $K$.

Fix $\epsilon>0$ and define
\begin{equation*}
v_{0, K}(x):=\epsilon \,u_{0,K}(x).
\end{equation*}
The preceding properties immediately imply the following.

\begin{itemize}

\item For every $K\geq1$, $v_{0,K}$ is real-valued and belongs to $H^\infty$;

\item Writing $\tilde{y}_{0,K}=(1-\partial^2_x)v_{0,K}$, we have 
\begin{equation}\label{ty0}
\tilde{y}_{0, K}\geq 0\ \text{on}\ \R, \qquad \tilde{y}_{0, K}(0)>0;
\end{equation}

\item Uniformly in $K$, 
\begin{equation}
\|v_{0, K}\|_{H^{5/2}}\lesssim \epsilon,\label{52v}
\end{equation}
and
\begin{equation}
\lim_{K\to\infty} v_{0, K}'(0)\tilde{y}_{0, K}(0)=-\infty,\label{limv}
\end{equation}
while, for all sufficiently large $K$,
\begin{equation}
\left(v_{0, K}'(0)\right)^2\tilde{y}_{0, K}(0)>\sqrt{2}\|v_{0,K}\|^3_{H^1}.\label{ineqv}
\end{equation}

\end{itemize}


\section{Main argument}

First, thanks to the properties established for $v_{0,K}$, in particular \eqref{ty0}, \eqref{limv}, and \eqref{ineqv}, we may apply Theorem \ref{blow-th} with
\[
(w_0, x_0)=(v_{0,K}, 0),
\]
provided $K$ is sufficiently large. It follows that the corresponding solution of the Cauchy problem \eqref{cp} blows up in finite time $T$, where 
\begin{equation*}
\aligned
T&\leq \frac{-w'_0(x_0)}{\sqrt{2}\|w_0\|^3_{H^1}}- \frac{1}{2}\sqrt{\left(\frac{\sqrt{2}w'_0(x_0)}{\|w_0\|^3_{H^1}}\right)^2-\frac{2\sqrt{2}}{\|w_0\|^3_{H^1}\,z_0(x_0)}}\\
&=\frac{\frac{1}{\sqrt{2}\|w_0\|^3_{H^1}\,z_0(x_0)}}{\frac{-w'_0(x_0)}{\sqrt{2}\|w_0\|^3_{H^1}}+ \frac{1}{2}\sqrt{\left(\frac{\sqrt{2}w'_0(x_0)}{\|w_0\|^3_{H^1}}\right)^2-\frac{2\sqrt{2}}{\|w_0\|^3_{H^1}\,z_0(x_0)}}}\\
&\leq \frac{\frac{1}{\sqrt{2}\|w_0\|^3_{H^1}\,z_0(x_0)}}{\frac{-w'_0(x_0)}{\sqrt{2}\|w_0\|^3_{H^1}}}=\frac{1}{-w'_0(x_0)z_0(x_0)}.
\endaligned
\end{equation*}
In view of  \eqref{limv}, we may increase $K$ further so that
\[
T<\epsilon.
\]  

We now fix this choice of $K$. This is the initial datum associated with the prescribed $\epsilon>0$ for which we establish Theorem \ref{main}. Since $v_{0,K}\in H^\infty$, Theorem \ref{wp-th} yields
\[
u\in C([0,T); H^\infty(\R)).
\]
By \eqref{52v}, we have
\[
\|u_{0}\|_{H^{5/2}}=\|v_{0,K}\|_{H^{5/2}}\lesssim \epsilon.
\] 
Therefore, it remains only to prove that 
\[
\limsup_{t\to T-}\|u(t)\|_{H^{5/2}}=\infty.
\]

For this purpose, we first derive suitable energy estimates. Using the local formulation \eqref{floc} of the FORQ equation, we multiply it by $m$ and integrate by parts. Since $u\in C([0,T); H^\infty(\R))$, all computations are justified, and we obtain
\begin{equation*}
\frac{d}{dt}\int_{\R}m^2\,dx=-2\int_{\R} u_x\,m^3\,dx.
\end{equation*}
Differentiating \eqref{floc} with respect to $x$, multiplying by $m_x$, and integrating by parts similarly gives
\begin{equation*}
\frac{d}{dt}\int_{\R}m_x^2\,dx= \frac{4}{3}\int_{\R} u_x\,m^3\,dx- 10\int_{\R} u_x\,m\,m_x^2\,dx.
\end{equation*}
Adding these identities yields
\begin{equation}\aligned
\frac{d}{dt}\left\{\|m(t)\|^2_{H^1}\right\}&=-\int_{\R} u_x\,m\left(\frac{2}{3}\,m^2+10\,m_x^2\right)dx\\&\leq 10 \|u_x(t)m(t)\|_{L^\infty}\|m(t)\|^2_{H^1}.
\endaligned\label{y-en}
\end{equation}
Since
\begin{equation}
\|m(t)\|_{H^1}\sim \|u(t)\|_{H^3},
\label{y-u}
\end{equation}
this estimate is well suited for controlling the $H^3$ norm.

It therefore remains to estimate $\|u_x(t)m(t)\|_{L^\infty}$. By Lemma \ref{wx-li}, together with the fact that $B^1_{\infty,\infty}(\R)$ and $H^2(\R)$ are Banach algebras, 
\begin{equation*}
\aligned
\|u_x(t)&m(t)\|_{L^\infty}\lesssim \|(u^2(t))_x\|_{L^\infty}+\|(u_x^2(t))_x\|_{L^\infty}\\ 
&\lesssim
1+\|u^2(t)\|_{B^1_{\infty,\infty}}\log_2(2+\|u(t)^2\|^2_{H^2})+\|u_x^2(t)\|_{B^1_{\infty,\infty}}\log_2(2+\|u^2_x(t)\|^2_{H^2})\\
&\lesssim
1+\|u(t)\|^2_{B^1_{\infty,\infty}}\log_2(2+\|u(t)\|^4_{H^2})+\|u_x(t)\|^2_{B^1_{\infty,\infty}}\log_2(2+\|u_x(t)\|^4_{H^2})\\
&\lesssim
\left(1+\|u(t)\|^2_{B^2_{\infty,\infty}}\right)\log_2(2+\|u(t)\|^2_{H^3}).
\endaligned
\end{equation*}

Combining this estimate with \eqref{y-en} and \eqref{y-u} gives
\begin{equation}
\aligned
\frac{d}{dt}\left\{\|m(t)\|^2_{H^1}\right\}&\lesssim \left(1+\|u(t)\|^2_{B^2_{\infty,\infty}}\right)\log_2(2+\|u(t)\|^2_{H^3})\|m(t)\|^2_{H^1}\\
&\sim \left(1+\|u(t)\|^2_{B^2_{\infty,\infty}}\right)\log_2(2+\|m(t)\|^2_{H^1})\|m(t)\|^2_{H^1}.
\endaligned
\label{diff-ineq}
\end{equation}

We now claim that
\begin{equation}
\limsup_{t\to T-}\|u(t)\|_{B^2_{\infty,\infty}}=\infty.\label{b1ii}
\end{equation}
Suppose instead that this fails. Then there exists a constant $M=M(\epsilon)$  such that
\[
\|u(t)\|_{B^2_{\infty,\infty}}\leq M, \qquad \forall\,t\in [0,T).
\] 
Substituting this bound into \eqref{diff-ineq} yields
\begin{equation*}
\frac{d}{dt}\left\{\|m(t)\|^2_{H^1}\right\}\leq C(1+M^2)\log_2(2+\|m(t)\|^2_{H^1})\|m(t)\|^2_{H^1}, \qquad \forall\,t\in [0,T).
\end{equation*}
Applying Lemma \ref{gron} and using \eqref{y-u}, we conclude that 
\[
 \sup_{t\in [0,T)}\|u(t)\|_{H^3}\sim \sup_{t\in [0,T)}\|m(t)\|_{H^1}<\infty.
\]
This contradicts the blow-up criterion in Theorem \ref{wp-th}, according to which
\[
\limsup_{t\to T-}\|u(t)\|_{H^3}=\infty.
\]
Hence \eqref{b1ii} follows.
 
Finally, the Sobolev embedding for Besov spaces gives
\[
\|w\|_{B^2_{\infty,\infty}}\lesssim \|w\|_{B^{5/2}_{2,2}}\sim\|w\|_{H^{5/2}}.
\]
Combing this with \eqref{b1ii} immediately yields
\[
\limsup_{t\to T-}\|u(t)\|_{H^{5/2}}=\infty,
\]
which completes the proof of Theorem \ref{main}.


\section*{Acknowledgements}
The author is deeply grateful to God and the Holy Theotokos for Their love, patience, and mercy, which inspired him while working on this paper. 

He is also thankful to Alex Himonas for his friendship and for introducing and guiding him to work on Camassa-Holm-type equations. 

The author used OpenAI's ChatGPT during the exploratory phase of this work to assist in generating and refining candidate initial data for the equation under study. 

\bibliographystyle{amsplain}
\bibliography{bNE}

\end{document}